\documentclass[12pt]{article}
\usepackage{amsmath,amssymb,amsthm,color} 
\usepackage[unicode,breaklinks=true,colorlinks=true]{hyperref}
\usepackage{tikz} %
\numberwithin{equation}{section}
\newtheorem{theorem}{Theorem}[section]

\newtheorem{lemma}[theorem]{Lemma}
\newtheorem{definition}[theorem]{Definition} 
\newtheorem{corollary}[theorem]{Corollary}
\newtheorem{assumption}[theorem]{Assumption}
\newtheorem{problem}[theorem]{Open Problem}

\theoremstyle{remark}
\newtheorem{remark}[theorem]{Remark}

\newcommand{\mat}[1]{\begin{bmatrix} #1 \end{bmatrix}}
\newcommand{\bke}[1]{\left( #1 \right)}

\newcommand{\bket}[1]{\left\{ #1 \right\}}
\newcommand{\norm}[1]{\| #1 \|}

\newcommand{\bka}[1]{\left\langle #1 \right\rangle}

\newcommand{\al}{\alpha}

\newcommand{\de}{\delta}
\newcommand{\e}{\epsilon}
\newcommand{\ga}{{\gamma}}

\newcommand{\la}{\lambda}

\newcommand{\Om}{{\Omega}}
\newcommand{\om}{{\omega}}
\newcommand{\si}{\sigma}

\renewcommand{\th}{\theta}
\renewcommand{\Re}{\mathop{\mathrm{Re}}}

\newcommand{\R}{{\mathbb R}}\newcommand{\RR}{{\mathbb R }}
\newcommand{\N}{{\mathbb N}}
\newcommand{\ZZ}{{\mathbb Z}}
\newcommand{\cR}{{\mathcal R}}
\newcommand{\pd}{{\partial}}
\newcommand{\nb}{{\nabla}}
\newcommand{\lec}{\lesssim}

\newcommand{\I}{\infty}
 
\renewcommand{\div}{\mathop{\mathrm{div}}\nolimits}

\def\si{\sigma}
\newcommand{\donothing}[1]{{}}

\newcommand{\EQ}[1]{\begin{equation}\begin{split} #1 \end{split}\end{equation}}

\makeatletter
\newcommand{\xRightarrow}[2][]{\ext@arrow 0359\Rightarrowfill@{#1}{#2}}
\makeatother

\newcommand{\phase}{\phi}

\begin{document} \title{Rotationally corrected scaling invariant solutions to the Navier-Stokes equations}
\author{Zachary Bradshaw and Tai-Peng Tsai}
\date{}
\maketitle

\begin{abstract}
We introduce new classes of solutions to the three dimensional Navier-Stokes equations in the whole and half spaces that add rotational correction to self-similar and discretely self-similar solutions.  We construct forward solutions in these new classes for arbitrarily large initial data in $L^3_w$ on the whole and half spaces.  We also comment on the backward case.

\end{abstract}

\nocite{*}

\section{Introduction}
The 3D Navier-Stokes equations in a domain $\Omega \subset \R^3$ read
\begin{align}
\begin{array}{ll}\label{eq:NSE}
 \partial_t v -\Delta v +v\cdot\nabla v+\nabla \pi  = 0%
\\  \nabla\cdot v = 0%
\end{array}
\quad \mbox{~in~}\Omega\times (0,\infty),
\end{align}
and are supplemented with the initial condition 
\[
v|_{t=0}=v_0,
\]
where $v_0:\Omega\to \R^3$ is given and satisfies $\nabla\cdot v_0 = 0$.
If $\Omega$ possesses a boundary $\partial \Omega$ with outernormal $\nu$, then we require
\[
v|_{\partial \Omega}=0,\qquad {v_0}\cdot \nu|_{\partial \Omega}=0.
\]
In this paper, $\Om$ is either the whole space $\R^3$ or the half space $\R^3_+= \{ x=(x_1,x_2,x_3)\in \R^3, x_3>0\}$.

Solutions to \eqref{eq:NSE} satisfy a natural scaling: given a solution $v$ and $\lambda>0$, it follows that 
\begin{equation}
\label{1.2}
	v^{\lambda}(x,t)=\lambda v(\lambda x,\lambda^2 t),
\end{equation}
is also a solution with associated pressure 
\begin{equation}
	\pi^{\lambda}(x,t)=\lambda^2 \pi(\lambda x,\lambda^2 t),
\end{equation}
and initial data 
\begin{equation}
v_0^{\lambda}(x)=\lambda v_0(\lambda x).
\end{equation}

Leray introduced self-similar solutions to \eqref{eq:NSE} in \cite{leray}.
A solution is \emph{self-similar} (SS) if 
\begin{equation}
\label{eq1.5}
v^\lambda(x,t)=v(x,t)
\end{equation}
for every $\lambda>0$.
The solutions considered by Leray are called \emph{backward} since they are defined for $-\infty<t<0$.  We will consider \emph{forward} solutions defined for $0<t<\infty$.
If the scale invariance \eqref{eq1.5} holds for a particular $\lambda>1$, not necessarily for every $\lambda>1$,
then we say $v$ is \emph{discretely self-similar} with factor $\lambda$, i.e.~$v$ is $\lambda$-DSS. The initial data $v_0$ is SS or $\lambda$-DSS if the appropriate scaling invariance holds with the time variable omitted.

The existence for SS/DSS solutions for \emph{small} initial data follows from
the unique existence theory of \emph{mild solutions} in various scaling invariant functional spaces, see \cite{GiMi,Kato,Barraza,CP,Koch-Tataru}.
The theory for large data is more 
recent: Jia and Sverak established the first large data existence result in \cite{JiaSverak} for self-similar data which is H\"older continuous on $\R^3\setminus 0$.   It is based on a priori H\"older estimates near initial time for \emph{local Leray solutions}  introduced by Lemari\'e-Rieusset in \cite{LR} (see also \cite{KiSe}), and is extended by Tsai \cite{Tsai-DSSI} to construct $\la$-DSS solutions under the assumption that $\la$ is close to $1$, or if the data is axisymmetric with no swirl.
A second construction is obtained by Korobkov and Tsai \cite{KT-SSHS}, which is valid in the half space as well as the
whole space, and is based on the a priori $H^1$ estimate obtained by Leray's method of contradiction and the triviality of $H^1_0$-solutions of Euler equations. 
Note that the first construction \cite{JiaSverak,Tsai-DSSI} does not work in the half space, while the second construction \cite{KT-SSHS} does not work for DSS solutions. A third construction of the authors \cite{BT1} constructs SS and $\la$-DSS solutions for any data in $L^3_w$ (i.e., weak $L^3$ defined in \eqref{weakL3}) and, in the DSS case, any $\la>1$. It is based on a new a priori energy estimate, particular to the associated Leray equations to be introduced in \eqref{eq:Leray} (not available to Navier-Stokes equations \eqref{eq:NSE}), and is a weak solution theory using the Galerkin approximation, not the Leray-Schauder theorem used in \cite{JiaSverak,Tsai-DSSI,KT-SSHS}. Although it is stated only for the whole space in 
\cite{BT1}, its method works also for the half space, as will be made apparent as a special case of this article.

The purpose of this article is to introduce and investigate a new class of solutions with scaling properties resembling those of SS or DSS solutions \emph{modulo rotational corrections}. There is a rich literature on fluids surrounding \emph{rotating obstacles}, see the survey \cite{Hishida}. For ease of notation,
we will only consider rotations around the $x_3$-axis with matrices
\[
R_s=R( s) = \mat{ 
\cos (s) & -\sin (s) & 0 \\
\sin (s) & \cos ( s) & 0 \\
0 & 0 & 1 }.\]
Note $R( s)R(\tau)=R(\tau)R( s)$ for any $s,\tau \in \R$, and
\[
\frac d{ds} R(s) = JR(s) = R(s)J, \quad 
J= \mat{ 
0 & -1 & 0 \\
1&0 & 0 \\
0 & 0 & 0 }.\]

A vector field $v(x,t)$ is said to be \emph{rotated self-similar} (RSS) if, for some fixed $\alpha\in \R$ and for \emph{all} $\la>0$,
\begin{equation}
\label{v-RSS}
v(x,t) = \la R(- 2\alpha \log \la)\, v\! \bke{\la R(2 \alpha \log \la) x, \la^2 t}, \quad \forall x,\forall t, \forall \la.
\end{equation}
The constant $\al$ will be called the \emph{angular speed}, relative to the new time variable $s$ to be defined in \eqref{variables}.
An RSS vector field is always DSS with any factor $\la>1 $ such that $2 \alpha \log \la \in 2\pi \mathbb{Z}$.
When $\alpha=0$ it becomes SS. The choice $\th(\la)=2\alpha \log \la $ in the argument of $R(\cdot)$ is natural because $\la>0$ is arbitrary and hence we need
\begin{equation}
\th(\la)+ \th(\mu) = \th(\la \mu),\quad \forall \la, \mu>0.
\end{equation}
Setting $\la=t^{-1/2}$, the RSS vector field $v$ satisfies
\EQ{
v(x,t) = R(\al \log {t}) \frac 1{\sqrt{t}}\, v\!\bke{R(-\al \log {t}) \frac x{\sqrt{t}} , 1}, \quad \forall x, \forall t.
} 
Thus the value of $v$ is determined by its value at any fixed time, and given any profile at a fixed time we can construct an RSS vector field.

A vector field $v(x,t)$ is said to be \emph{rotated discretely self-similar} (RDSS) if, for some $\la>1$ (not necessarily all $\la>1$) 
and some $\phase \in \R$, 
\begin{equation}
\label{v-RDSS}
v(x,t) = \la R(-\phase)\, v\! \bke{\la R(\phase) x, \la^2 t}, \quad \forall x,\forall t.
\end{equation}
We call $\la$ the \emph{factor} and $\phase$ the \emph{phase}.
When $\phase\in 2\pi \mathbb{Z}$ we recover $\la$-DSS vector fields. 
If $n\phase=2\pi m$ for some integers $n>0$ and $m$, then $v$ is DSS with factor $\la^n$. If $\frac\phase{2\pi}$ is irrational, in general $v$ is not DSS. 
For any $t>0$ let $\tau(t)\in [1,\lambda^2)$ satisfy $\tau=\lambda^{2k}t$ for some $k\in \ZZ$.  Then,
\[
v(x,t)= \la^k R(-k\phase)\, v\! \bke{\la^k R(k \phase) x, \tau},
\]
i.e., $v$ is decided entirely by its values on $t\in [1,\lambda^2)$.
{Note that an RSS vector field with angular speed $\al$ is always RDSS for any factor $\la>1$ with phase $\phase=2\al \log \la $.}

In summary, the inclusions between these classes are
\[
\mbox{SS}\subsetneq \mbox{RSS} \subsetneq \mbox{DSS} \subsetneq \mbox{RDSS}.
\]
 
Similar to SS/DSS vector fields,
these vector fields are also called \emph{forward} if they are defined for $0<t<\infty$, or  \emph{backward} if they are defined for $-\infty<t<0$. They are called 
\emph{stationary} if they are time-independent.

A vector field
$v_0(x):\R^3\to \R^3$ is RSS if for some $\alpha\in \R$,
\begin{equation}
\label{v0-RSS}
v_0(x) = \la R(- 2 \alpha \log \la) v_0\bke{\la R( 2\alpha \log \la) x}, \quad \forall x, \forall \la,
\end{equation}
and is
RDSS if for some $\la>1$ and some $\phase \in \R$, 
\begin{equation}
\label{v0-RDSS}
v_0(x) = \la R(-\phase) v_0\bke{\la R(\phase) x}, \quad \forall x.
\end{equation}
Setting $\la=|x|^{-1}$, an RSS $v_0$ satisfies
\begin{equation}
\label{v0-RSS2}
v_0(x)=\frac 1 {|x|} R(2\alpha \log  {|x|}) \,v_0\!\bke{ R(-2\alpha \log  {|x|}) \frac x{|x|}}.
\end{equation}
Thus the value of $v_0$ is determined by its values on the unit sphere. Similarly, if $v_0$ is RDSS, then it is determined by its values on $\{ x: 1\leq |x|<\la \}$.
Clearly, if $\lim _{t \searrow 0} v(x,t)  = v_0(x)$ and $v(x,t) $ satisfies \eqref{v-RSS} (resp.~\eqref{v-RDSS}), then $v_0(x)$ satisfies \eqref{v0-RSS} (resp.~\eqref{v0-RDSS}).

The ansatz of RSS solutions was originally proposed by Grisha Perelman for backward solutions defined for $-\infty<t<0$  to Seregin around a decade ago (private communication of G.~Seregin). See the Appendix for details.   We are not aware of any previous study of RDSS solutions.

Our goal in this paper is to construct RSS/RDSS solutions for general RSS/RDSS initial data.
One needs to verify that there is an abundance of nontrivial such $v_0$.
If $v_0$ is axisymmetric (i.e., $v_0(x) = R(-s) v_0(R(s)x)$ for any $s$), then RSS is reduced to SS, and RDSS  is reduced to DSS. 
It is relatively easy to construct non-axisymmetric RDSS vector fields, in the same way as the DSS case: One can choose any divergence free  vector field with compact support in the annulus $B_\la \setminus \bar B_1$ (or its intersection with $\R^3_+$), and extend its definition to entire $\R^3$ (or $\R^3_+$) by RDSS property \eqref{v0-RDSS}. To construct non-axisymmetric RSS vector fields, we use 
the spherical
coordinates $\rho,\phi,\th$ with basis vectors
\begin{equation}
\label{spherical}
e_\rho =\frac {x}{\rho}, \quad
e_\phi = \left(\frac {x_1x_3}{r\rho},\frac {x_2x_3}{r\rho},-\frac r\rho \right),
\quad
e_\th = \left(-\frac {x_2}r,\frac {x_1}r,0\right),
\end{equation}
where $\rho=|x|$ and $r=\sqrt{x_1^2+x_2^2}$,
and consider vector fields of the form at $\rho=1$:
\begin{equation}
v_0(1, \phi,\th)= \frac {f(\phi,\th)}{\sin \phi} e_\rho + \frac {g(\phi,\th)}{
  \sin \phi}e_\phi + {h(\phi,\th)} e_\th .
\end{equation}
The RSS condition  \eqref{v0-RSS} gives
\begin{equation}
v_0 (\rho,\phi,\th)= \frac {f(\phi,\th_\rho)}{\rho \sin \phi} e_\rho + \frac {g(\phi,\th_\rho)}{\rho
  \sin \phi}e_\phi + \frac {h(\phi,\th_\rho)}{\rho} e_\th ,
\end{equation}
where $\th_\rho = \th-\al \ln \rho$.
We impose that  $f, g, h$ are $2\pi$-periodic in $\th$ and vanish sufficient order at $\phi=0,\pi$ or at $\phi=0,\pi/2$.
The divergence-free condition $\div v_0=0$ becomes 
\begin{equation}
(1-\al \pd_\th) f + \pd_\phi g + \pd_\th h = 0.
\end{equation}
To get nontrivial dependence on $\th$, we may impose the
$k$-\emph{equivariance} ansatz for $k \in \N$:
\[
f=\Re F(\phi) e^{ik\th}, \quad
g=\Re G(\phi) e^{ik\th}, \quad
h=\Re H(\phi) e^{ik\th}, 
\]
and it suffices to choose complex-valued smooth functions $F,G,H$ of $\phi\in (0,\pi)$ (or
$\phi\in (0,\pi/2)$) that satisfy
\begin{equation}
(1-ik \al ) F + G' + ik H = 0
\end{equation}
and vanish sufficient order at $\phi=0,\pi$ (or at $\phi=0,\pi/2$).

These solutions can be better understood in similarity variables, introduced by Giga-Kohn \cite{Giga-Kohn}. Consider the similarity transform
\begin{equation}
\label{similarity transform}
v(x,t) = \frac 1{\sqrt{t}} V(z,s), \quad \pi(x,t) = \frac 1{{t}} \Pi(z,s), 
\end{equation}
where
\EQ{
z=\frac x{\sqrt{t}},\quad s=\log t.
\label{variables}
}
For a cone-like domain $\Om$, the system of Navier-Stokes equations \eqref{eq:NSE} in $\Om \times (0,\infty)$ is equivalent to the time dependent \emph{forward Leray equations}
\EQ{\label{eq:Leray}
 \partial_s V  - \frac 12 V - \frac z2\cdot\nabla_z V -\Delta_z V+V\cdot\nabla_z V+\nabla_z \Pi  = 0,
 \quad \nabla_z\cdot V = 0
}
in $\Om \times (-\infty,\infty)$. 
 When $\pd \Om$ is nonempty, the boundary condition $v|_{\pd \Om}=0$ corresponds to 
\EQ{\label{eq:Leray-BC}
V|_{\pd \Om}=0.
}
For backward solutions defined for $-\infty<t<0$, we replace $t$ by $-t$ in \eqref{similarity transform} and \eqref{variables}, and get the backward 
Leray equations
\EQ{\label{eq:Leray-}
 \partial_s V  + \frac 12 V + \frac z2\cdot\nabla_z V -\Delta_z V+V\cdot\nabla_z V+\nabla_z \Pi  = 0,
 \quad \nabla_z\cdot V = 0.
}
A SS solution $v(x,t)$ of \eqref{eq:NSE} corresponds to a stationary solution of \eqref{eq:Leray} or 
\eqref{eq:Leray-}. A DSS solution $v(x,t)$ of \eqref{eq:NSE} with factor $\la>1$ corresponds to an $s$-periodic solution $V(z,s)$ of \eqref{eq:Leray} or 
\eqref{eq:Leray-} with period $2\log \la$.
An RSS solution $v(x,t)$ of \eqref{eq:NSE} satisfying \eqref{v-RSS} with angular speed $\al$ corresponds to a solution of \eqref{eq:Leray} or 
\eqref{eq:Leray-} satisfying (with $\tau= 2\log \la$)
\EQ{\label{V-RSS}
V(z,s) = R(-\al \tau ) V (R(\al \tau ) z, s+\tau), \quad \forall  \tau\in \R.
}
Finally, an RDSS solution $v(x,t)$ of \eqref{eq:NSE} satisfying \eqref{v-RDSS} with factor $\la$ and phase $\phase$ corresponds to a solution of \eqref{eq:Leray} or 
\eqref{eq:Leray-} satisfying
\EQ{\label{V-RDSS}
V(z,s) = R(-\phase) V(R(\phase)z,s+2\log \la).
}
Note we get \eqref{V-RDSS} from \eqref{V-RSS} by choosing $\tau= 2\log \la$ and $\phase =\al \tau$.

The initial condition $v|_{t=0}=v_0$ does not have a clear meaning for $V$ in general, but corresponds to a ``boundary condition'' for $V(z,s)$ at spatial infinity for SS/DSS/RSS/RDSS solutions, see \cite{KT-SSHS,BT1} and \S \ref{sec.Leray}.

Leray \cite{leray} proposed the SS solution as a possible ansatz for singular solutions and gave
the stationary case of \eqref{eq:Leray-}. His original problem of existence of $V(z) \in W^{1,2}(\R^3) \subset L^3(\R^3)$  was excluded in Ne\v cas, {R\accent23 u\v {z}i\v {c}ka}, and {\v Sver\'ak} in \cite{NRS}. It was later extended to exclude $V \in L^q(\R^3)$, $3<q\le \infty$, or $v \in L^{10/3}(B_1 \times (-1,0))$, by \cite{Tsai-ARMA}.
Giga and Kohn \cite{Giga-Kohn} were aware of the correspondence between \eqref{eq:NSE} and \eqref{eq:Leray-} and used the corresponding similarity transform to study the singularity of nonlinear heat equations.

If we assume  
\begin{equation}
\label{V.u.eq}
V(z,s) = R_\theta u(y,s), \quad  \Pi(z,s) = p (y,s),\quad y = R_\theta^T z,
\end{equation}
with
\[
R_\theta = R(\th(s))
\]
for some function $\th(s)$, then \eqref{eq:Leray} is 
equivalent to %
\EQ{\label{eq:Leray-rot}
  \partial_s u  + \dot \th J u - \dot \th (J y)\cdot \nb u
  &-  \frac 12 u -  \frac y2\cdot\nabla u -\Delta_y u +u\cdot\nabla u+\nabla p  = 0,
 \\ & \nabla \cdot u = 0,
}
where $\nb = \nb_y$ and $\Delta = \Delta_y$. If $\Om=\R^3_+$, then $v|_{\pd \Om}=0$ implies 
\[
u|_{\pd \Om}=0.
\]
To illustrate these observations consider the case when $w= V-z=R q$, $q=u-y$, with $R=R_\theta$.  Then
\EQ{
(w \cdot \nb_z  V)_i &= R_{kl} q_l(y)  \pd_{z_k} R_{ij} u_j(R^T z) = R_{kl} q_l(y)R_{ij}  \pd_{y_m}  u_j(y) R_{km}. 
}
Since $R \in O(3)$,  $R_{kl}R_{km} = \de_{lm}$ and hence
\EQ{
(w \cdot \nb_z  V)_i &= \de_{lm} q_l(y)R_{ij}  \pd_{y_m}  u_j(y)   = R_{ij} q_l (y) \pd_{y_l}  u_j(y)
= [R (q \cdot \nb_y u)]_i.
}

If $v(x,t)$ is an RSS solution of \eqref{eq:NSE} satisfying \eqref{v-RSS}, then $V(z,s)$ satisfies  \eqref{V-RSS}, and hence $u(y,s)$ is a stationary solution of\eqref{eq:Leray-rot} with constant $\dot \theta = \alpha$.
On the other hand, for any RDSS solution $v(x,t)$ of \eqref{eq:NSE} satisfying \eqref{v-RDSS} with factor $\lambda>1$ and phase $\phase$, 
$V(z,s)$ satisfies  \eqref{V-RDSS}.
Let 
\begin{equation}
\label{alpha-choice}
T=2 \log \la,\quad
 \al_k= \frac {2k\pi + \phase}T ,
\end{equation}
for an arbitrary integer $k\in \mathbb{Z}$. Then $v(x,t)$ 
corresponds to a periodic solution $u(y,s)$ of \eqref{eq:Leray-rot} with constant $\dot \theta = \alpha_k$ and period $T$. To be definite we will take $\al=\al_0 = \frac \phase T$.

To construct solutions in Theorem \ref{thrm:existence}
using our method, the system \eqref{eq:Leray-rot} needs not be autonomous but needs to be periodic in $s$. However, we have not been able to find applications of non-constant $\dot \theta$, hence we will let $\dot \theta =\alpha$ be 
constant in the rest of the paper for simplicity of presentation.

The natural spaces to study $v_0$ and $v$ as described above are, respectively, $L^3_{w}(\RR^3)$ and $L^\infty((0,\infty);L^3_w(\RR^3))$. Recall that $f\in L^3_w(\Om)$ if and only if
$\|f\|_{L^3_w(\Om)}<\I$,
where 
\begin{equation}\label{weakL3}
\|f\|_{L^3_w(\Om)} = \sup_{s>0}\,s \,m(f,s)^{1/3},
\end{equation}
and $m(f,s)$ is the distribution function of $f$ given by
\[
m(f,s)=|\{ x\in \Omega: | 	f(x)|>s  \}|.
\]  
Let $L^3_{w,\si}(\Omega)$ be the subspace of $L^3_w(\Omega, \R^3)$ of divergence free vector fields which satisfy $v_0\cdot \nu|_{\partial\Om}=0$ if $\partial\Om$ is nonempty and has the unit outer normal vector field $\nu$.

Since $L^3_w(\RR^3)$ embeds continuously into the space of uniformly locally square integrable functions $L^2_{u\,loc}(\RR^3)$, one may construct global-in-time \emph{local Leray solutions} for our data as in  \cite{JiaSverak,Tsai-DSSI, BT1} in the whole space. However, because this paper aims to construct solutions on both the whole and half spaces, and
 there is presently no existence theory for {local Leray solutions} on the half space, we only construct weak solutions.

\begin{definition}[EP-solutions to \eqref{eq:NSE}]Let $\Omega$ be a domain in $\R^3$. The vector field $v$ defined on $\Omega\times (0,\I)$ is an \emph{energy perturbed solution to \eqref{eq:NSE}} -- i.e.~an \emph{EP-solution} -- with initial data $v_0\in L^3_{w,\si}(\Omega)$ if \begin{equation}\label{eq:NSEweak} 
\int_0^\I  \big( (v,\partial_s f)-(\nabla v,\nabla f)- v\cdot\nabla v ,f)  \big)  \,ds =0,
\end{equation}
 for all $f\in  \{ f\in C_0^\I(\Om\times \R_+):\nabla\cdot f = 0 \}$, if
\[
v-Sv_0 \in L^\infty(0,T;L^2(\Omega))\cap L^2(0,T;H^1(\Omega)),
\]
for any $T>0$, and if
\[
\lim_{t\to 0^+} \| v(t)-Sv_0(t) \|_{L^2(\Omega)} = 0,
\]
where $Sv_0(t)\in L^\I(0,\I;L^3_{w,\si}(\Omega))$ is the solution to the time-dependent Stokes system with initial data $v_0$, see \S\ref{sec.Stokes}.
\end{definition}

\begin{remark} The name ``energy perturbed solution'' means that the difference $v-Sv_0$ is in the energy class, although $Sv_0$ is not.
We do not mention the pressure in Definition \ref{def:periodicweaksolutionR3}. Note that a pressure can be constructed after the fact since $Sv_0$ has an associated pressure and $v-Sv_0\in L^2$ for all positive times. 
\end{remark}

\begin{theorem} \label{thrm:existence}
Assume $v_0$ is in $L^3_{w,\si}(\Omega)$ where $\Omega\in \{ \R^3,\R^3_+ \}$.  

(i) (RSS) If $v_0$ is RSS, satisfying \eqref{v0-RSS} for some angular speed $\alpha \in \R$, then there exists an EP-solution $v$ 
on $\Omega\times [0,\infty)$ with initial data $v_0$, which  is RSS and satisfies \eqref{v-RSS} for the same $\alpha$. It satisfies $v|_{\pd \Om}=0$ if $\Om= \R^3_+$.

(ii) (RDSS) If $v_0$ is RDSS, satisfying \eqref{v0-RDSS} for some factor $\la>1$ and phase $\phase \in \R$, 
then there exists an EP-solution $v$ on $\Omega\times [0,\infty)$ with initial data $v_0$,  which is RDSS and satisfies \eqref{v-RDSS} for the same $\la$ and $\phase$. It satisfies $v|_{\pd \Om}=0$ if $\Om= \R^3_+$.
\end{theorem}

\emph{Comments on Theorem \ref{thrm:existence}}
\begin{itemize}
\item If $\al=0$ then the class of RDSS solutions coincides with the class of $\lambda$-DSS solutions defined in \cite{Tsai-DSSI,BT1} (where they were only considered on the whole space). Theorem \ref{thrm:existence} therefore provides a construction of $\la$-DSS solutions on the half-space for any divergence free $\la$-DSS initial data belonging to $L^3_{w,\sigma}(\Omega)$.

\item If $v_0$ is RSS then it is RDSS for any $\lambda>1$ and thus there exist EP-solutions $v_\la$ to the 3D NSE on $\Omega\times [0,\infty)$ which are RDSS.  Letting $\la\to 1$ we can obtain a solution $v$ which is RSS. This procedure mimics that given in \cite[Section 5.1]{BT1} and we omit the details.
\end{itemize}

In our proof we directly construct a solution to the rotated Leray equations \eqref{eq:Leray-rot}.  To do this we perturb $u(y,s)$ by subtracting the image $U_0(y,s)$ of the solution to the Stokes equations under the rotated self-similar transform, i.e.~we seek a solution of the form $U=u-U_0$.  Essentially, we are treating $U_0$ as the boundary data at spatial infinity of $u$.  Fortunately, $U\in L^2$ (which is untrue for both $u$ and $U_0$).  To get formal a priori estimates for $U$ via energy methods we develop new bounds for $U_0$ using self-similarity and Solonnikov's formulas in $\R^3_+$ \cite{Solonnikov} (in \cite{BT1} we were working with the solution to the heat equation in $\R^3$ which was easier to bound). Unfortunately, $U_0$ does not give us the needed a priori bound since we don't have
\[
\int_{\R^3} (U\cdot \nabla U_0)\cdot U \,dy\leq \gamma \| U\|_{H^1}^2 ,
\]
where $0<\gamma <1$. The idea is to replace $U_0$ by an asymptotically similar profile $W$ which allows the above estimate.  In the whole space case \cite{BT1}, we used a non-compact correction involving singular integrals to ensure $W$ was divergence free. This does not work when there are boundaries.  To get around this we construct another profile $W$ using the Bogovskii map \cite{Bog}.

On the other hand, since we do not seek local Leray solutions our argument is shorter than in \cite{BT1}.  In particular, we do not need to use mollifiers to obtain the local energy inequality when we are constructing $U$ via a Galerkin scheme.  We also do not need a priori bounds for the pressure; these were only used in \cite{BT1} to obtain the local energy inequality.  
 
\begin{remark}\label{remark1.4}
Note that, in the whole space $\R^3$ case, we can recover pressure estimate and local energy inequality in Theorem \ref{thrm:existence}, in the same way as in \cite{BT1}.
\end{remark}

\medskip \noindent \emph{Notation.}\quad We will use the following function spaces on a domain $\Om\subset \R^3$:
\begin{align*}
&\mathcal V=\{f\in C_0^\infty({ \Omega;\R^3}) ,\, \nabla \cdot f=0 \},
\\& X = \mbox{the closure of~$\mathcal V$~in~$H_0^1(\Omega)$} ,
\\& H = \mbox{the closure of~$\mathcal V$~in~$L^2(\Omega)$},
\end{align*}where $H_0^1(\Omega)$ is the closure of $C_0^\infty(\Omega)$ in the Sobolev space $H^1(\Omega)$.  Let $X^*(\Omega)$ denote the dual space of $X(\Omega)$. 
Let $(\cdot,\cdot)$ be the $L^2(\Omega)$ inner product and $\langle\cdot,\cdot\rangle$ be the dual product for $H^1$ and its dual space $H^{-1}$, or that for $X$ and $X^*$.
 
\medskip \noindent \emph{Organization.}\quad In Section 2 we construct solutions to a rotationally corrected Leray system.  In Section 3 we study RDSS solutions to the Stokes equations.  Section 4 contains the proof of Theorem \ref{thrm:existence} which uses the results of Sections 2 and 3.  Finally, in Section 5 the Appendix, we give comments on the backward case.

\section{An auxiliary problem in similarity variables}\label{sec.Leray}

In this section we study a time periodic weak solution to the auxiliary problem
\begin{equation}
\label{eq:aux}
\begin{array}{ll}
	\partial_s u +\alpha J u  - \alpha Jy \cdot \nabla u
	-\Delta u=\frac 12 u+\frac 12 y\cdot \nabla u -\nabla p -u\cdot\nabla u	&\mbox{~in~}\Omega \times \R
	\\  \nabla\cdot u = 0  &\mbox{~in~}\Omega\times \R
	\\ u = 0  &\mbox{~on~}\partial\Omega \times \R 
	\\ 	\displaystyle \lim_{|y_0|\to\infty} \int_{B_1(y_0)\cap \Omega} |u(y,s)-U_0(y,s)|^2\,dx= 0& \mbox{~for all~}s\in \R
	\\  u(\cdot,s)=u(\cdot, s+T) &\mbox{~for all~}s\in \R,
\end{array}
\end{equation}where $\Omega \in \{\R^3,\R^3_+  \}$ and $U_0(y,s)$ is a given $T$-periodic divergence free vector field defined on $\Omega$ which vanishes on $\partial\Omega$.  If $\Om=\R^3$ we ignore the boundary condition on $\pd \Om$. 
Our goal is to construct a solution $u$ satisfying the problem in the weak sense,~i.e.
\begin{equation}\label{u.eq-weakR}
\int_\R  \big( (u,\partial_s f)-(\nabla u,\nabla f)+(
   \alpha Jy \cdot \nabla u-\al J u
 + \frac 12 u
+\frac 12 y\cdot\nabla u-u\cdot\nabla u ,f)  \big)  \,ds =0,
\end{equation}
holds for all divergence free $f\in C_0^\I( \Om\times \R)$.

In our application we require that $U_0$ additionally satisfies the following assumption.

\begin{assumption} \label{AU_0}
The vector field $U_0(y,s) :\Omega \times \R \to \R^3$ is periodic in $s$ with period $T>0$, divergence free, vanishes on $\partial\Omega$, and satisfies, for some $q\in (3,\infty]$,
\begin{itemize}
\item for all divergence free $f\in C_0^\I( \Om\times \R)$,
\begin{align}
\int_\R  \big( (U_0,\partial_s f)-(\nabla U_0,\nabla f)+(
   \alpha Jy \cdot \nabla U_0-\al J U_0
 + \frac {U_0}2 
+\frac y2 \cdot\nabla U_0,f)  \big)  \,ds =0, 
\label{U0.eq}
\end{align}
\item the inclusions:
\begin{align*}
& U_0\in L^\infty (0,T;L^4\cap L^q(\Omega )), %
\\& \partial_s U_0 \in L^\infty(0,T;L^{6/5}_{loc}(\overline \Omega)),
\quad { \nabla U_0\in L^2(0,T;L_{loc}^{2}(\overline \Omega)),}
\end{align*}
\item the decay estimate:
\begin{equation} \label{decay estimate}
\sup_{s\in [0,T]}\|U_0  \|_{L^q(\Omega\setminus B_R)}\leq \Theta(R),
\end{equation}
for some $\Theta:[0,\infty)\to [0,\infty)$ such that $\Theta(R)\to 0$ as $R\to\infty$.
\end{itemize} 
\end{assumption}

The decay estimate \eqref{decay estimate} ensures the existence of a good revised asymptotic profile in Lemma \ref{lemma:W}. The assumption $U_0\in L^\infty (0,T;L^4(\Omega))$ implies in particular $U_0 \cdot \nb U_0 \in L^2 (0,T;H^{-1}(\Omega))$, which is essential for the a priori bound in Lemma \ref{lemma:Galerkin}. 

Periodic weak solutions to \eqref{eq:aux} %
are defined as follows.

\begin{definition}[Periodic weak solutions to \eqref{eq:aux}]
\label{def:periodicweaksolutionR3} 
Let $U_0$ satisfy Assumption \ref{AU_0}. 
The field $u$ is a periodic weak solution to \eqref{eq:aux} if it is divergence free,
{if $u|_{\pd \Om}=0$,} if {$u(s)=u(s+T)$} for all $s\in \R$, if
\begin{equation}\notag
U:= 
u-U_0\in L^\infty(0,T;L^2(\Omega))\cap L^2(0,T;H^1(\Omega)),
 \end{equation} 
and if $u$ satisfies \eqref{u.eq-weakR} for all divergence free $f\in C_0^\I(\Om\times \R)$.
\end{definition}

The main result of this section is the following theorem.

\begin{theorem}[Existence of periodic weak solutions to \eqref{eq:aux}]\label{thrm.aux}
Let $\Omega\in \{\R^3,\R^3_+\}$ and assume $U_0:\Omega\times \R\to \R^3$ satisfies Assumption \ref{AU_0} with $q=10/3$. Then, there exists a periodic weak solution to \eqref{eq:aux} corresponding to $U_0$ in the sense of Definition \ref
{def:periodicweaksolutionR3}.
\end{theorem}

The choice $q=10/3$ was chosen for Remark \ref{remark1.4}, for the convenience of the proof of the local energy inequality, see \cite{BT1}. Otherwise we can take any $q \in (3,\infty]$.

To prove Theorem \ref{thrm:existence} we seek a solution of the form $u=U+U_0$ as this homogenizes the boundary condition at spatial infinity.  This leads to a source term in the perturbed equation that is not necessarily small.  To get around this we replace $U_0$ by $W$ which eliminates the possibly large behavior of $U_0$ near the origin, with the correction $W-U_0$ being compactly supported. This will give us the crucial bound,
 \begin{equation}
   \int (f\cdot\nabla f)\cdot W \,dy \leq \delta ||f||_{H_0^1(\Om)}^2,
\end{equation}
where $\delta$ is a given small parameter.

Fix $Z\in C^\infty(\Om)$ with $0 \le Z \le 1$, $Z(y)=1$ for $|y|>1$, and $Z(y)=0$ for $|y|<1/2$. This can be done so that $|\nb Z|+|\nb^2 Z| \lesssim 1$.  Fix $r>1$ and let $U_r(y)=U_0(ry)$.  Let $\hat U_r(y)=Z(y)U_r(y)$.  Then, 
\[
\nb\cdot \hat U_r=U_r\cdot \nb Z.
\]
Since this is non-zero we will need a correction term obtained by Bogovskii's construction from \cite{Bog} which we now recall.

\begin{lemma}\label{lemma.bogovskii}
Let $K$ be a bounded Lipschitz domain in $\R^3$.  There is a linear map $\Phi$ such that for any scalar $f\in C_0^\I(K)$ with $\int_{K}f\,dx=0$, we have $\Phi f \in  C_0^\I(K)$, 
\[
\nb \cdot \Phi f = f,
\]
and,
\[
 \|\Phi f \|_{W^{1,q}(K)}\leq c(q,K) \|f \|_{L^q(K)	},
\]
for any $1<q<\I$.  Thus $\Phi $ can be extended to a bounded map from $\{f\in L^q(K): \int_{K}f\,dx=0 \}$ to $W_0^{1,q}(K)$.
\end{lemma}

Let $K=B_1\cap \Omega$.  Since $U_0$ is divergence free we have,
\[
\int U_r\cdot \nb Z \,dx=0,
\]
and can thus apply Lemma \ref{lemma.bogovskii} and let $w_r:=\Phi ( U_r\cdot \nb Z)$. 
Then,
\[
\nabla\cdot w_r =-U_r\cdot \nb Z,
\] 
and, furthermore,
\EQ{
\label{wr.est}
  \|w_r\|_{W^{1,q}(B_1)}\leq c(q, B_1)\|U_r\cdot \nb Z\|_{L^q}
  \leq c(q, B_1,Z)\|U_r\|_{L^q}. 
}

Let $\hat U_0(ry)=\hat U_r(y)$ and $w(ry)=w_r(y)$.  
Our replacement for $U_0$ is, 
\[
W:=\hat U_0 + w, \quad W(y) = U_0(y) Z(y/r) + w_r(y/r).
\]

The following notation is convenient.  For a given $F(y,s)$ and any $\zeta \in C^1_0(\Omega)$, denote
\begin{equation}%
\notag
LF = 	\partial_s F+\alpha JF   - \alpha Jy \cdot \nabla F
-\Delta F-\frac 12F-\frac 12y\cdot \nabla F ,
\end{equation}
and
\begin{equation}
\label{LW.def}
\bka{LF,\zeta} =(\partial_s F+\al J F   - \alpha Jy \cdot \nabla F
-\frac 12F-\frac 12y\cdot \nabla F,\zeta) + (\nabla F, \nabla \zeta).
\end{equation}

\begin{lemma}[Revised asymptotic profile]
\label{lemma:W} Let $\Omega\in \{\R^3,\R^3_+ \}$.
Fix $q\in (3,\infty]$ and suppose $U_0$ satisfies Assumption \ref{AU_0} for this $q$. For any small $\de>0$, let $W$, $\hat U_0$, and $w$ be defined as above with $r=r_0(\de)$ sufficiently large.  Then $W$ is $T$-periodic and divergence free, 
\begin{equation}\label{energy-perturbation}
U_0 - W \in L^\infty(0,T; L^2(\Omega)) \cap L^2(0,T; H^1(\Omega)) ,
\end{equation} 
\begin{equation}\label{Wsmall.est}
\|W\|_{L^\infty(0,T;L^q(\Omega))}\leq \de, 
\end{equation} 
\begin{equation}\label{WL4.est}
\norm{W}_{L^\infty(0,T;L^4(\Omega))}\leq c(r_0,U_0),
\end{equation}
and
\begin{equation}
\label{LW.est}
\norm{LW}_{L^\infty(0,T; H^{-1}(\Omega))} \leq c(r_0,U_0), %
\end{equation}
where $c(r_0,U_0)$ depends on $r_0$ and quantities associated with $U_0$ which are finite by Assumption \ref{AU_0}.
\end{lemma}

\begin{proof} 
$T$-periodicity in $s$ follows from the fact that $U_0$ is $T$-periodic.  $W$ is divergence free since $\nabla\cdot w_r =-U_r\cdot \nb Z$.

To see \eqref{energy-perturbation}, recall that $U_0-W= (U_0-\hat U_0)-w$.  
Both $U_0-\hat U_0$ and $w$ are supported in $K$. Since $w\in W_0^{1,q}(K)$ by \eqref{wr.est}, we have $w\in L^\I(0,T;H^1(\Omega))$. The difference $U_0-\hat U_0$ is in $L^\I L^2 \cap L^2 H^1$  by Assumption \ref{AU_0}.  Therefore, we have \eqref{energy-perturbation}.

We now get refined estimates for $w$ in $W^{1,q}$ using \eqref{wr.est}: For any $r>1$ we have
\begin{align}
\int_{B_r} |w(y)|^q\,dx
&=\int_{B_1} |  w(rz)|^qr^3\,dz \nonumber
\\&= r^{3} \int_{B_1}|  w_r(z)|^q\,dz \nonumber
\\&\leq c(q,B_1,Z)r^{3} \int_{B_1\setminus B_{1/2}} |U_0(zr)|^q\,dz \nonumber
\\&= c(q,B_1,Z)\int_{B_r\setminus B_{r/2}} |U_0(y)|^q\,dy,
\label{w.est}
\end{align}
and
\begin{align}
\int_{B_r} |\nabla w(y)|^q\,dy
&= \int_{B_1} |r^{-1} \nb_z w(rz)|^qr^3\,dz \nonumber
\\&=r^{3-q} \int_{B_1}|\nb_z w_r(z)|^q\,dz \nonumber
\\&\leq c(q,B_1,Z)r^{3-q} \int_{B_1\setminus B_{1/2}} |U_0(zr)|^q\,dz \nonumber
\\&= c(q,B_1,Z)r^{-q} \int_{B_r\setminus B_{r/2}} |U_0(y)|^q\,dy.
\label{Dw.est}
\end{align}
Note that the constants above do not depend on $r$.  

Since $W=\hat U_0 + w$, by \eqref{w.est} and the definition of $\Theta(r)$,
\[
\|W\|_{L^\infty(0,T;L^q(\Omega))} \le \Theta(r_0/2) + c(q,B_1,Z)^{1/q} \Theta(r_0/2)
\le \de ,
\]
for $r_0$ sufficiently large. This show \eqref{Wsmall.est}.

To prove \eqref{WL4.est}, we first establish a pointwise estimate for $w$,
\[
|w(y)| \lec r^{1-3/q} \norm{\nabla w}_{L^q(B_r)} +  r^{-3/q} \norm{w}_{L^q(B_r)}.
\]
Since $w$ is compactly supported this implies that $w\in L^\I(0,T;L^4(\Om) )$.  By assumption \ref{AU_0}, $U_0\in L^\I(0,T;L^4(\Om))$.  Therefore, $W\in L^\I(0,T;L^4(\Om))$ also.

We now prove \eqref{LW.est}.  First observe that, since $\partial_s w_r = \Phi(\pd_s U_r \cdot \nabla Z)$, by \eqref{Dw.est} we have
\[
 \| \partial_s w\|_{L^2(\Omega)}\lesssim 
 \| \nabla \partial_s w\|_{L^{6/5}(\Omega)}\lesssim 
r^{-1} \|  \partial_s U_0\|_{L^{6/5}(B_r)},
\]
which is finite by Assumption \ref{AU_0}.
Since $w$ is compactly supported, the $W^{1,q}$ estimates imply $w\in H^1$.  It follows that $\| \alpha Jw   - \alpha Jy \cdot \nabla W
-\Delta w-w-y\cdot \nabla w \|_{H^{-1}}<\infty$.  Hence $Lw\in L^\I(0,T;H^{-1})$.  Let $Z_*(y)=Z(-y/r_0)$ so that $\hat U_0(y)= U_0(y) Z_*(y)$. Then,
\[
LW=(LU_0)Z_*+Lw -\al J y\cdot\nb Z_* U_0 -2\nb Z_*\cdot\nb U_0-U_0\Delta Z_* - y\cdot\nb Z_* U_0. 
\]
By Assumption \ref{AU_0}, $LU_0=0$.  Since $Z_*$ is compactly supported, Assumption \ref{AU_0} implies that $\al J y\cdot\nb Z_* U_0 -2\nb Z_*\cdot\nb U_0-U_0\Delta Z_* - y\cdot\nb Z_* U_0\in L^2$.  Hence $LW\in L^\I(0,T;H^{-1})$.
\end{proof}

We seek a solution to \eqref{eq:aux} of the form $u=W+ U$, where $W$ is as in Lemma \ref{lemma:W} for $\de = 1/4$ for a given $U_0$ satisfying Assumption \ref{AU_0}.  
Let
\begin{equation}\label{RW.def}
\mathcal{R}(W) := 	LW + W\cdot\nabla W ,
\end{equation}
where  $L$ is defined by \eqref{LW.def}.
The weak formulation for $U$ is
\begin{align}\label{perturbed-Leray}
\frac d {ds}(U,f)
&=	
		- (\nabla U,\nabla f) 
		+ (-\al J U   + \alpha Jy \cdot \nabla U +\frac 12 U+\frac 12 y\cdot \nabla U, f)
\\ 
\notag&\quad - (U \cdot\nabla U, f)
		-(W\cdot\nabla U+U\cdot \nabla W,f)-\langle \mathcal{R}(W),f\rangle,
\end{align}
and holds for all $f \in \mathcal V$ and a.e.~$s\in (0,T)$.

We use the Galerkin method as in \cite{BT1}.
Let $\{a_{k}\}_{k\in \N}\subset \mathcal V$ be an orthonormal basis of $H$.
For a fixed $k$, we look for an approximation solution of the form $U_k(y,s)= \sum_{i=1}^k b_{ki}(s)a_i(y)$.
We first prove the existence of and \emph{a priori} bounds for $T$-periodic solutions $b_k=(b_{k1},\ldots,b_{kk})$ to the system of ODEs
\begin{align}\label{eq:ODE}
\frac d {ds} b_{kj} = & \sum_{i=1}^k A_{ij}b_{ki} +\sum_{i,l=1}^k B_{ilj} b_{ki}b_{kl} +C_j,%
\end{align}
for $j\in \{1,\ldots,k\}$,
where
\begin{align}
\notag A_{ij}&=- (\nabla a_{i},\nabla a_j) 
		+ (-\al J a_i + \al Jy \cdot \nb a_i + \frac 12 a_i+\frac y2 \cdot \nabla a_i, a_j) 
\\\notag &\quad\,	 -(a_i\cdot \nabla W,a_j)
		- (W\cdot\nabla a_i, a_j)
\\\notag B_{ilj}&=- (a_i \cdot\nabla a_l, a_j)
\\\notag C_j&=-\langle \mathcal R (W),a_j\rangle.
\end{align}
 
\begin{lemma}[Construction of Galerkin approximations]\label{lemma:Galerkin} Fix $T>0$ and let $W$ satisfy the conclusions of Lemma \ref{lemma:W} with $\de=\frac 14$.
\begin{enumerate}
\item For any $k\in \mathbb N$, the system of ODEs \eqref{eq:ODE} has a $T$-periodic solution $b_{k}\in H^1(0,T)$.
\item Letting
\begin{equation} \notag
U_k(y,s)=\sum_{i=1}^k b_{ki}(s)a_i(y),
\end{equation}we have 
\begin{equation}\label{ineq:uniformink}
||U_k||_{L^\infty (0,T;L^2(\Omega))} + ||U_k||_{L^2(0,T;H^1(\Omega))}<C,
\end{equation}where $C$ is independent of   $k$.
\end{enumerate}
\end{lemma}

\begin{proof} Our proof is nearly identical to \cite[Proof of Lemma 2.6]{BT1}.
Fix $k\in \N$. For any  $U^{0}\in \operatorname{span}(a_1,\ldots,a_k)$, 
there exist $b_{kj}(s)$ uniquely solving \eqref{eq:ODE} with initial value $b_{kj}(0)=(U^{0},a_j)$, and belonging to $H^1(0,\tilde T)$ for some time $0<\tilde T\leq T$. Based on \cite[Proof of Lemma 2.6]{BT1} we may assume $\tilde T = T$.
Multiply the $j$-th equation of \eqref{eq:ODE} by $b_{kj}$ and sum to obtain,
\begin{align} \notag
&\frac 1 2 \frac d {ds} ||U_k||_{L^2}^2 + \frac 1 2 ||U_k||_{L^2}^2+ ||\nabla U_k||_{L^2}^2
\\&\leq - ( U_k\cdot \nabla W, U_k ) - \langle \mathcal{R}(W), U_k\rangle-(\al J U_k  - \alpha Jy \cdot \nabla U_k  ,U_k).\label{ineq:1}
\end{align}

We now bound the right hand side of the above inequality.  From the definition of $J$ we have 
\[
(\al J U_k    - \alpha Jy \cdot \nabla U_k ,U_k)=0.
\]
By \eqref{Wsmall.est} and $\div U_k=0$, we have
\begin{equation}
  \big| ( U_k\cdot \nabla W, U_k )  \big| \leq  \frac 1 8 ||U_k||_{H^1}^2 .
\end{equation}
Recall that $\cR( W)= LW + \div(W\otimes W)$.  Then,
\begin{equation} \label{ineq:2}
|(\cR(  W),U_k)| \le (\norm{LW}_{H^{-1}} +   \|W  \|_{L^4}^2 ) \norm{U_k}_{H^1}  \le C_2+ \frac 1 8 ||U_k||_{H^1}^2 .%
\end{equation}
where $C_{2}=C(\norm{LW}_{H^{-1}} +  \|W  \|_{L^4}^2)^2 $ is independent of $s$, $T$, and $k$.

Using Lemma \ref{lemma:W}, the estimates \eqref{ineq:1}--\eqref{ineq:2} imply
\begin{equation}  \label{ineq:kenergyevolution}
	 \frac d {ds} ||U_k||_{L^2}^2
	 +   \frac 1 2 ||U_k||_{L^2}^2
	 +   \frac 1 2 ||\nabla U_k||_{L^2}^2 \leq C_{2} .
\end{equation}
The Gronwall inequality implies
\begin{equation} \label{ineq:gronwall}
\begin{split}
e^{s/2} ||U_k(s)||_{L^2}^2
&\leq ||U^{0}||_{L^2}^2 +  \int_0^{  T} e^{\tau/2}  C_2 \,dt
\\
& \le  ||U^{0}||_{L^2}^2 + e^{T/2} C_2 T
\end{split}
\end{equation}
for all $s\in [0, T]$.

By \eqref{ineq:gronwall} we can choose $\rho>0$ (independent of $k$) so that 
\begin{equation}\notag
 ||U^{0}||_{L^2}\leq \rho \Rightarrow ||U_{k}(T)||_{L^2}\leq \rho.
\end{equation}
The mapping $\mathcal T:B_\rho^k\to B_\rho^k$ given by $\mathcal T(b_{k}(0))=b_k(T)$, where $ B_\rho^k$ is the closed ball of radius $\rho$ in $\R^k$, is continuous. Thus $\mathcal T$ has a fixed point by the Brouwer fixed-point theorem, i.e.~there exists some $U^{0}\in \operatorname{span}(a_1,\ldots,a_k)$ so that $b_k(0)=b_k(T)$. 

It remains to check that \eqref{ineq:uniformink} holds. The $L^\infty L^2$ bound follows from \eqref{ineq:gronwall}
since $\norm{U^0}_{L^2} \le \rho$, which is independent of $k$.
 Integrating  \eqref{ineq:kenergyevolution} over $ [0,T]$ and using $U_k(0)=U_k(T)$, we get
\begin{equation} \label{eq2.33}
 \frac 1 2 \int_0^T \big(||U_k||_{L^2}^2
+  ||\nabla U_k||_{L^2}^2 \big)\,dt \le C_2 T
\end{equation}
which gives an upper bound for $\| U_k  \|_{L^2(0,T;H^1 )}$ which is uniform in $k$.
\end{proof}

\begin{proof}[Proof of Theorem \ref{thrm.aux}]

By standard arguments, there exists $U\in {L^2(0,T;H_0^1(\R^3))}$ and a subsequence of $\{U_k\}$ (which we still index with $k$) so that 
\begin{align*}
& U_k\rightarrow U \mbox{~weakly in}~L^2(0,T;X),
\\& U_k\rightarrow U  \mbox{~strongly in}~L^2(0,T;L^2(K))  \mbox{~for all compact sets~}K\subset \bar\Omega,
\\& U_k(s)\rightarrow U (s) \mbox{~weakly in}~L^2 \mbox{~for all}~s\in [0,T].
\end{align*}
The weak convergence guarantees that $U (0)=U (T)$, %
and that $U$ satisfies \eqref{perturbed-Leray}.

Let $u=U+W$. Since $W$ and $U$ are $T$-periodic, so is $u$. That $u$ satisfies \eqref{u.eq-weakR} follows from \eqref{perturbed-Leray} and integrating in time.  The \emph{a priori} bounds for $u-W$ extend to bounds for $u-U_0$ since $U_0-W \in L^\infty(0,T;L^2(\Omega))\cap L^2(0,T;H^1(\Omega))$.  This implies that $u$ is a periodic weak solution in the sense of Definition \ref{def:periodicweaksolutionR3}.  
\end{proof}

\section{The non-stationary Stokes system} \label{sec.Stokes}

In this section we study the non-stationary Stokes system in $\Om \in \{ \R^3,\R^3_+\}$. The result will be used to verify Assumption \ref{AU_0} for suitable initial data and $U_0$ defined by \eqref{def:U0} in \S\ref{sec.proof}. The focus is on the case $\Om=\R^3_+$ since the Stokes system is reduced to the heat equation when $\Om=\R^3$.

The non-stationary Stokes system is 
\begin{align}
\begin{array}{ll}\label{eq:Stokes}
 \partial_t v -\Delta v +\nabla \pi  = 0&\mbox{~in~}\Omega \times [0,\infty)
\\  \nabla\cdot v = 0&\mbox{~in~}\Omega\times [0,\infty),
\end{array}
\end{align}
and is required to satisfy the initial condition
\[
v|_{t=0}=v_0,
\]
where $v_0\in L^3_{w,\sigma}(\Omega)$ is given.  
If $\Omega=\R^3_+$ then we augment \eqref{eq:Stokes} with the boundary condition
\begin{align}
\label{eq:Stokes2}
v|_{x_3=0}=0.
\end{align}

Let $S$ denote the solution operator for the Stokes system \eqref{eq:Stokes}--\eqref{eq:Stokes2} on $\Omega$.  If $\Omega=\R^3$ then $S$ is just the solution operator for the heat equation, i.e.,
\[
Sv_0(x,t)=e^{t\Delta}v_0(x) =\int_{\R^3}T(x,y,t)  v_0(y)\,dy,
\]
where
\[
T(x,y,t)=	\frac C { t^{3/2}} \, e^{-\frac{|x-y|^2}{4t}} ,\quad C=(4\pi)^{-3/2}.
\]

If $\Omega=\R^3_+$ the formula for $S$ is more complicated.  In  \cite{Solonnikov}, Solonnikov showed that if $v_0$ is divergence free and vanishes on $\partial \Omega$, then
\[
(Sv_0)_i(x,t)=\int_{\R^3}G_{i,j}(x,y,t)v_{0j}(y)\,dy,
\]
for
\[
G_{i,j}(x,y,t)
=
\delta_{i,j}T(x,y,t) + G_{i,j}^*(x,y,t), 
\]
where 
\begin{align*}
G_{i,j}^*(x,y,t) =&-\delta_{i,j}T(x,y^*,t)
\\&-4(1-\delta_{j,3})\frac \partial{\partial x_j}\int_{\R^2\times [0,x_3]}\frac \partial {\partial x_i} E(x-z)T(z,y^*,t)\,dz,
\end{align*}
and $E(x)=(4\pi |x|)^{-1}$ is the fundamental solution to the Laplace equation on $\R^3$, and for a given $y=(y_1,y_2,y_3)\in \R^3$ we denote 
\[
y' = (y_1, y_2),\quad y^*=(y_1,y_2,-y_3).
\]
Moreover, $G_{ij}^*$
satisfies the pointwise bound (\cite[(2.38)]{Solonnikov}
\begin{equation}
\label{Solonnikov.est}
|\pd_t^ s D_x ^m D_y^\ell G_{ij}^*(x,y,t)|\lec t^{-s - \ell_3/2}
(\sqrt t+x_3)^{-m_3} (\sqrt t+|x-y^*|)^{-3- |m'|- |\ell'|} e^{-\frac{c
    y_3^2}t}
\end{equation}
for all $s\in \N_0=\bket{0,1,2,\ldots}$ and $m,\ell \in \N_0^3$. Note that the $L^p\to L^q$ bound of $S(t)$ in \cite{Ukai} is not sufficient for our purpose.

We first observe that, when $v_0$ is RDSS for given factor $\la>1$ and phase $\phase \in \R$  (i.e., $v_0$ satisfies \eqref{v0-RDSS}),
$Sv_0$ is RDSS for the same values of $\la$ and $\phi$, (i.e., $Sv_0$ satisfies \eqref{v-RDSS}). To see this in the whole space case let $\xi=\la R(\phi) y$.  We have $|x-y|^2/(4t) = |\la R(\phi)x -\xi|^2/(4\la^2 t)$ and  $d\xi =\la^3 dy$.  Therefore,
\begin{align*}
Sv_0(x,t)&= \lambda R(-\phi ) \int_{\RR^3}\frac C {(\lambda^2 t)^{3/2}}e^{|\lambda R(\phi)x-\xi|^2/(4\lambda^2t)} v_0(\xi)\,d\xi
\\&= \lambda R(-\phi ) Sv_0(R(\phi)\lambda x,\lambda^2t).
\end{align*}For the half space case note that the above scaling goes through if $T(x,y,t)$ is replaced by $T(x,y^*,t)$.  For the remaining part of $G_{i,j}^*$ we carry out a change of variables letting $\chi=\la R(\phi)x$, $\xi=\la R(\phi) y$, and $\zeta=\la R(\phi)z$ and obtain
\begin{align*}
&\int_{\R^3_+}\frac \partial{\partial x_j}\int_{\R^2\times [0,x_3]}\frac \partial {\partial x_i} E(x-z)T(z,y^*,t)\,dz\, v_{0j}(y) \,dy
\\&=C\lambda R(-\phi) \int_{\R^3_+} \frac \partial{\partial x_j} \int_{\R^2\times [0,x_3]}  \frac {\partial}{ \partial x_i} \frac 1 {|x-z|} \frac 1 {(\la^2 t)^{3/2}} e^{-|z-y^*|^2/4t} v_{0j}(\xi)\,dz\,d\xi
\\&=C\lambda R(-\phi)\int_{\R^3_+} \frac \partial{\partial \chi_j} \int_{\R^2\times [0,\chi_3]}  \frac {\partial}{ \partial \chi_i} \frac 1 {|\chi-\xi|} \frac 1 {(\la^2t)^{3/2}} e^{-|\zeta-\xi^*|^2/4\la^2 t} v_{0j}(\xi)\,d\zeta\,d\chi.
\end{align*}
Thus, $Sv_0$ also satisfies the desired scaling \eqref{v-RDSS} when $\Om =\R^3_+$.

If $v_0$ is RSS with angular speed $\al$ (i.e., $v_0$ satisfies \eqref{v0-RSS}), then it satisfies \eqref{v0-RDSS} for any $\la>1$ with $\phi = 2 \al \log \la$. By the above, $Sv_0$ is also RDSS with the same $\la$ and $\phi$. Thus $Sv_0$ is RSS and
satisfies \eqref{v-RSS} with the same $\al$.

\medskip
  
In \cite[Lemma 3.1]{BT1}, functions in $L^3_w(\R^3)\cap \mathrm{DSS}$ are shown to belong to $L^3_{loc}(\R^3\setminus \{ 0\})$.  The next lemma extends this to RDSS functions on the whole or half space.

\begin{lemma}\label{lemma:equivalence}Assume  $\Omega\in \{\R^3,\R^3_+  \}$.
If  $f$ is defined in $\Omega$ and is RDSS with factor $\la>1$ and phase $\phase \in \R$, then  $f \in L^3_{loc}(\Omega \setminus \{0\})$ if and only if $f \in L^{3}_w(\Omega)$. 
\end{lemma}
\begin{proof} 
Recall
the distribution function for $f$ is
\[
m(\si,f) = |\{ x\in \Omega: |f(x)| > \si \}|,
\]
and the well known identity
\[
\int_\Omega |f|^p \,dx = p \int_0^\I \si^{p} m(\si,f) d \si/\si,
\]
which holds for $1\le p < \I$. Let $A_r = \bket{x \in \Om: r \le |x| < r \lambda }$, and
decompose $f$ as $f = \sum_{k \in \mathbb{Z}} f_k$ where $f_k(x) = f(x)$ if $x \in A_{\lambda^k }$ and $f_k(x)=0$ otherwise. For $\beta>0$, 
\EQ{\label{eq3.2}
m(\beta, f) &= \sum_{k \in \mathbb{Z}} m(\beta,f_k)
= \sum_{k \in \mathbb{Z}} m(\lambda^k \beta,f_0) \lambda^{3k},
}
where we have used the scaling property 
\EQ{
f_k (x) = \lambda^{-k}R(k \phi) f_0(\lambda^{-k}R(-k \phi) x),
}
since $f$ is RDSS. This says rotational corrections do not effect the scaling of the distribution functions $m$ when compared to the DSS case considered in  \cite[Lemma 3.1]{BT1}. Indeed, \eqref{eq3.2} matches \cite[Equation (3.2)]{BT1} and the ensuing steps from \cite[Proof of Lemma 3.1]{BT1} apply.  We thus omit the remaining details.  
\end{proof}

The next several lemmas concern the integrability and decay of solutions to the Stokes system where the initial data is RDSS and belongs to $L^3_{w,\si}(\Om)$.

\begin{lemma}\label{lemma:V0bound} Assume  $\Omega\in \{\R^3,\R^3_+  \}$.
Suppose  {$v_0 \in L^3_{w,\sigma}(\Omega)$ is  RDSS with factor $\la>1$ and phase $\phase \in \R$.}  Then, for all $s\in\N_0$ and $m \in \N_0^3$, and for all $R>0$, we have
\begin{align} \label{eq:lemma3.2}
\sup_{t\geq 1} \| \partial_t^s \nb^m_x Sv_0 \|_{L^\I(\Omega\cap B_R)} <\I.
\end{align}
\end{lemma}

\begin{proof}
This is trivial in the whole space -- see \cite{BT1}.  Assume $\Om=\R^3_+$.  
 
Let $A_{\lambda^k}=\{x\in \Om: \la^k\leq |x|<\la^{k+1}  \} $.  Assume $x\in B_{R}$ for some $R>1$.  Choose $k_0$  sufficiently large so that $2R< \la^{k_0}$.
 Recall that 
\[
(Sv_0)_i(x,t) = \int_{\R^3_+}  G_{i,j}(x,y,t)v_{0j}(y)  \,dy,
\]
where $G_{i,j}= -\delta_{i,j} T(x,z,t) + G_{i,j}^*$.
By the Solonnikov estimate \eqref{Solonnikov.est}, we have
\begin{align}\label{ineq.Sol}
\sup_{t>1} \big|\partial_t^s \nb^m_x G_{i,j}(x,y,t) \big| \le \frac {C_{s,m}} {(1+|x-y|)^{3}}.
\end{align}

Denote $g(y) = \partial_t^s \nb^m_x G_{i,j}(x,y,t)$. 
For $k\geq k_0$ and $y \in A^{\la^k}$, we have $\big| g(y) \big| \lesssim |y|^{-3}$ and
\[
\int_{A^{\la^k}}| g(y)  |^{3/2}\,dy \lesssim \la^{-3k/2}.
\]
Now, 
\begin{align*}
\bigg|\int_{\R^3_+}  g(y)v_{0j}(y) \,dy \bigg|
&\leq 
\|g\|_{L^{2}(B_{\la^{k_0}})} \| v_0\|_{L^{2}(B_{\la^{k_0}})} 
\\&\quad +
\sum_{k \ge k_0}  \|g\|_{L^{3/2}(A_{\la^k})} \| v_0\|_{L^3(A_{\la^k})}
\\&\leq \|g\|_{L^{2}(B_{\la^{k_0}})} \| v_0\|_{L^{2}(B_{\la^{k_0}})} 
 +C \|v_0\|_{L^3(A_{1})} \sum_{k \ge k_0} \la^{-k},
\end{align*}
which is finite by Lemma \ref{lemma:equivalence} because $v_0\in L^3_w$.  
The above shows \eqref{eq:lemma3.2}.
\end{proof}

\begin{lemma}\label{lemma:V0decay}Assume  $\Omega\in \{\R^3,\R^3_+  \}$.
Suppose  {$v_0 \in L^3_{w,\sigma}(\Omega)$ is  RDSS with factor $\la>1$ and phase $\phase \in \R$.} Then,  
 \[
 \label{w1.decay}
 \sup_{1\leq t\leq \lambda^2}\norm{Sv_0(t)}_{L^{q}(|x|>r)}\leq \Theta(r),
 \]
for any $q \in (3,\I]$ and $t\in [1,\lambda^2]$ where
$\Theta:[0,\infty)\to [0,\infty)$ depends on $q$ but satisfies $\Theta(r)\to 0$ as $r\to\infty$.
\end{lemma}
\begin{proof}If $\Omega=\R^3$ then the proof of \cite[Lemma 3.2]{BT1} applies here with only superficial changes and we omit the details.  We thus focus on the half space case.  

Assume $\Omega=\R^3_+$ and $v_0$ is as in the statement of the lemma.  
By Lemma \ref{lemma:equivalence} we have $v_0\in L^3_{loc}(\Omega\setminus \{0\})$. Let
\[
\om_1(l) = \sup_{1<|x_0|<\la} \bigg(\int_{B(x_0,l)\cap \Omega} |v_0(x)|^3\,dx\bigg)^{1/3},
\]
and
\[
\om_2(l) = \bigg(\int_{x\in \Omega;x_3<l;1\leq |x|\leq \la} |v_0(x)|^3\,dx\bigg)^{1/3}.
\]
Clearly $\om_1(l),\om_2(l) \to 0$ as $l \to 0$.

Let $A_r=\{ x\in \Omega : r\leq |x| <\lambda r \}$ and $A_r^*=\{ z: r/2\leq |z| <2\lambda r \}$.  We first estimate $\|Sv_0(t)\|_{L^q(A_r)}$.  Write  
\EQ{
 &(Sv_0)_i (x,t)
\\&= \bigg\{
		\int_{|y|< r/2} 
		+\int_{y\in A_r^*;|x-y|<\gamma r}
		+\int_{y\in A_r^*;|x-y|\geq \gamma r}  
		+\int_{2\lambda r\leq |y| }
		 \bigg\}
	\delta_{i,j}T(x,y,t) v_{0j}(y) \,dy
\\&~+
\bigg\{
		\int_{|y|< r/2} 
		+\int_{y\in A_r^*;y_3<\gamma r}
		+\int_{y\in A_r^*;y_3\geq \gamma r}  
		+\int_{2\lambda r\leq |y| }
		 \bigg\}
	G_{i,j}^*(x,y,t) v_{0j}(y) \,dy
\\&= I_0^r(x,t)+I_1^r(x,t)+I_2^r(x,t)+I_3^r(x,t)
\\&~  +J_0^r(x,t)+J_1^r(x,t)+J_2^r(x,t)+J_3^r(x,t),}
where $0< \gamma \ll 1$ is an as-of-yet unspecified parameter. We suppress the indexes $i$ and $j$ since they play no role in our estimates.

Fix $r>1$ and $(x,t)\in A_r\times [1,\lambda^2]$. Using the formula for $T$ it is easy to see that
\[
 |I_0^r(x,t)|+|I_3^r(x,t)|\lesssim e^{-cr^2}.
\]
We have by H\"older's inequality that
\[
|I_1^r(x,t)| \le C \norm{e^{-cx^2}}_{L^{3/2}(\R^3)} \norm{ v_0}_{L^3(B(x,\gamma r))} .
\]
By re-scaling $v_0$ we have 
\[
 \norm{ v_0}_{L^3(B(x,\gamma r))}\leq  \om_1(\gamma),
\]
and, therefore,
\[
|I_1^r(x,t)| \le C \om_1(\gamma).
\]
Also, by H\"older's inequality we have
\EQ{
|I_2^r(x,t)| & \le C \int _{A_r^*} e^{-c \gamma^2 r^2}  | v_0(z)| \,dz 
\\
&\le C e^{-c \gamma^2 r^2}  \norm{ v_0}_{L^3(A_r^*)}  \norm{ 1}_{L^{3/2}(A_r^*)} 
\le C e^{-c \gamma^2 r^2}  r^2.
}

The remaining integrals will be bounded using 
Solonnikov estimate \eqref{Solonnikov.est} with $s=|m|=|\ell|=0$,
\begin{equation}\label{Solon2}
\sup_{t \ge 1} | G_{ij}^*(x,y,t)|\lec \frac { e^{-{c  y_3^2}}}
{(1+x_3+y_3+|x'-y'|)^{3}}.
\end{equation}
We first bound $J_0^r(x)$. Since $|y|<r/2$ and $|x|\ge r$, we have
\EQ{\label{eq3.10}
x_3 +y_3+ |x'-y'| \gtrsim r.
}
Thus $|G^*_{i,j}(x,y,t)| \lesssim r^{-3}$ in the integrand of $J_0^r$. 
It follows that 
\[
 |J_0^r(x)|\leq \frac C {r^3} \int_{B_{r/2}(0)} |v_0(y)|\,dy
=  \frac C {r^3} \sum_{k=1}^\infty  \norm{v_0}_{L^1(A_{r/(2\la^k)})} .
\]
Note that
\[
\norm{v_0}_{L^1(A_{r/(2\la^k)})} \le C r^{2}  \norm{v_0}_{L^3(A_{r/(2\la^k)})} = Cr^{2},
 \]
 independent of $k$,
using the fact that $v_0$ is RDSS.
Thus
\[
 |J_0^r(x)|\lesssim r^{-1}.
\]

Similarly, for $J_3^r(x)$,  since $|y|>2\la r$ and $|x|\le \la r$, we also have
\eqref{eq3.10}. Thus $|G^*_{i,j}(x,y,t)| \lesssim r^{-3}$ in the integrand of $J_3^r$, and 
\[
 |J_3^r(x)|\leq \frac C {r^3} \int_{B_{2\la r}^C(0)} |v_0(y)|\,dy
=  \frac C {r^3} \sum_{k=1}^\infty  \norm{v_0}_{L^1(A_{2\la^k r})} 
\lesssim r^{-1}.
\]

To bound $J_1^r$ note that
\[
\frac {e^{-cy_3^2}} {(1+x_3+|x'-y'|)^3} \leq \frac {e^{-cy_3^2}} {(1+|x'-y'|)^3}\in L^{3/2}(\Omega),
\]
uniformly in $x \in A_r$.
Then, by 
\eqref{Solon2} 
and H\"older's inequality,
\[
|J_1^r(x)|\leq \bigg\| \frac {e^{-cy_3^2}} {(1+x_3+|x'-y'|)^3} \bigg\|_{L^{3/2}} \int_{y\in A_r^*;y_3<\gamma r}|v_0|^3\,dy \lesssim \omega_2(\gamma),
\]
where we have re-scaled $v_0$ to obtain the last inequality.

For $J_2^r$ we have by H\"older's inequality, the fact that $y_3\geq \gamma r$, and re-scaling that
\[
|J_2^r(x)|\lesssim e^{-c \gamma^2 r^2} r^2 \|  v_0\|_{L^3(A_1^*)}.
\]

Taken together, these estimates show that for $r>1$
\EQ{
\norm{Sv_0(t)}_{L^\I(A_r)} \le  C \om_1(\gamma) + C\om_2(\gamma) + C e^{-c \gamma^2 r^2}  r^2+ Ce^{-c r^2} + C r^{-1},
}
where the constants are independent of $r$ and $\gamma$. 
The above inequality is still valid if  $\lambda^kr$ replaces $r$ for $k\in \N$, indeed we have 
\EQ{
\norm{Sv_0}_{L^\I(A_{\la^kr})} \le  C \om_1(\ga)+C \om_2(\ga) + C e^{-c \ga^2 (\la^kr)^2}  (\la^kr)^2+ Ce^{-c (\la^kr)^2} + C \la^{-k}r^{-1}.
}
The right hand side is decreasing in $k$ for fixed $\ga$ and $r$ and we conclude that
\[
 \norm{Sv_0}_{L^\infty({|x|\geq r})}\leq 
\sup_{k\in \N} \norm{Sv_0}_{L^\I(A_{\lambda^kr})} \leq C \om_1(\ga)+C \om_2(\ga)  + C e^{-c \ga^2 r^2}  r^2+ Ce^{-c r^2} + C r^{-1}.
\]
If $q\in (3,\I)$ we have
\begin{align*}
\norm{Sv_0}_{L^q(|x|\geq r)} &\leq C \norm{Sv_0}_{L^\I(|x|\geq r)}^{1-3/q} \norm{Sv_0}_{L^{3}_w}^{3/q} \\&\leq C (  C \om_1(\ga)+C \om_2(\ga) +   Ce^{-c \ga^2 { r}^2}  { r}^2+ C e^{-c { r}^2} + C r^{-1})^{1-3/q} \norm{v_0}_{L^{3}_w}^{3/q}
.\end{align*}

We now construct $\Theta(r)$. Let $\epsilon_k=2^{-k}$ for $k\in \N$. For each $\epsilon_k$, choose $\gamma_k>0$ sufficiently small so that 
\[
C \om_1(\gamma_k)+C \om_2(\gamma_k) \leq \frac  {\e_k^{q/(q-3)}} {2 C^{q/(q-3)}\| v_0 \|_{L^{3}_w}^{3/(q-3)}}.
\] 
Then choose $r_k$ sufficiently large so that $r_k>r_{k-1}$ and
\[ C e^{-c \gamma_k^2 r_k^2}  r_k^2 + Ce^{-c r_k^2} + C r_k^{-1}\leq \frac  {\e_k^{q/(q-3)}} {2C^{q/(q-3)}\| v_0 \|_{L^{3}_w}^{3/(q-3)}}.\]
Finally, let
\[
\Theta (r)=
\begin{cases}
1 &\text{if } 0<r<r_1
\\ \epsilon_k &\text{if } r_k\leq r< r_{k+1}
\end{cases},
\]
which completes our proof.
\end{proof}

\begin{corollary}\label{corrolary}Assume  $\Omega\in \{\R^3,\R^3_+  \}$.
Suppose  {$v_0 \in L^3_{w,\sigma}(\Omega)$ is  RDSS with factor $\la>1$ and phase $\phase \in \R$.} Then $\sup_{1\leq t\leq \la^2} Sv_0(t)\in L^q(\Omega)$ for every $q>3$.
\end{corollary}
\begin{proof}
This follows immediately from Lemma \ref{lemma:V0bound} and Lemma \ref{lemma:V0decay}.
\end{proof}

\section{EP-solutions to Navier-Stokes equations} %
\label{sec.proof}

In this section we prove the main Theorem \ref{thrm:existence}.

\begin{proof}[Proof of Theorem \ref{thrm:existence}] 

We will focus on the RDSS case (ii).The RSS case (i) can be obtained either as limit of case (ii) as indicated in the comment after Theorem 
 \ref{thrm:existence}, following \cite[\S 5.1]{BT1}, or proved directly in the time-independent setting, following \cite[\S 5.2]{BT1}.
 
Assume $v_0$ is an RDSS divergence free vector field in $L^3_{w,\sigma}(\Omega)$, $\Omega\in \{ \R^3,\R^3_+ \}$, with
phase $\phase \in \R$ and factor $\la>1$. Let $T =2 \log \la$ and 
choose $\al= \frac {\phase}T$ as in \eqref{alpha-choice}. 
Let $x,t,z,s$ satisfy \eqref{variables} and let $y=R_\th^T z$ and $\theta=\al s$.  Let  
\begin{equation}
\label{def:U0}
U_0(y,s) = R_\th ^T {\sqrt {t}}  Sv_0(x,t).
\end{equation}
Compare this formula to \eqref{similarity transform} and \eqref{V.u.eq}.
We now check that $U_0$ satisfies Assumption \ref{AU_0} with (any) $q \in (3,\I]$. Since $v_0$ is divergence free and RDSS, $Sv_0(x, t)$ is the divergence free, RDSS solution to the Stokes system \eqref{eq:Stokes} for $(x,t)\in \Om\times (0,\infty)$. It follows that $U_0(y,s)$ is  $T$-periodic, divergence free, and satisfies 
\eqref{U0.eq}. %
Lemmas \ref{lemma:V0bound} and \ref{lemma:V0decay} and Corollary \ref{corrolary} guarantee that $U_0$ satisfies the function space inclusions in Assumption \ref{AU_0}.

Let $u$ be the time-periodic weak solution of the Leray equations described in Theorem \ref{thrm.aux} with $U_0$ defined by \eqref{def:U0} and $q=10/3$.
Let $v(x,t)= R_\th u(y,s)/\sqrt{t}$.
Then $v$ satisfies \eqref{eq:NSEweak}, the weak form of the Navier-Stokes equations. 

It remains to show $v$ is an EP-solution.  Observe that
\begin{equation} \notag  v-Sv_0 \in L^\infty(1,\lambda^2;L^2(\Om))\cap L^2(1,\lambda^2;H^1(\Om)). \end{equation}
The $\lambda$-DSS scaling property implies
\begin{equation}\notag 
||v(t)-Sv_0(t)||_{L^2(\Om)}^2\lesssim t^{1/2} \sup_{1\leq \tau\leq \lambda^2} ||v(\tau)-Sv_0(\tau )||_{L^2(\Om)}^2,
\quad \forall t>0,
\end{equation}
and
\begin{equation}\notag
\int_0^{\lambda^2} \int || \nabla (v(t)-Sv_0(t))||_{L^2 }^2\,dx\,dt \lesssim \bigg( \sum_{k=0}^\infty  \lambda^{-k}  \bigg) \int_1^{\lambda^2}\int || \nabla (v(t)-Sv_0(t))||_{L^2 }^2\,dx\,dt.
\end{equation}
It follows that
\begin{equation}\label{ineq:time0}  v-Sv_0 \in L^\infty(0,\lambda^2;L^2(\Om))\cap L^2(0,\lambda^2;H^1(\Om)). \end{equation}
By re-scaling, these bounds hold up to any finite time. Thus, $v$ is an EP-solution and the proof is complete.
\end{proof}

\section{Appendix: Comments on the backward case}

In this section we summarize known results and open problems in the backward case. 
For more details please see \cite{JST}.

The classes of self-similar, RSS, DSS and RDSS solutions have their analog in the backward case, i.e., solutions of \eqref{eq:NSE} for 
\[
-\infty<t<0.
\]
In the backward case, one replaces 
$t$ by $-t$ in the similarity transform
\eqref{similarity transform}--\eqref{variables} to set
\[
v(x,t) = \frac 1{\sqrt{-t}} V(z,s), \quad \pi(x,t) = \frac 1{{-t}} \Pi(z,s), 
\quad
z=\frac x{\sqrt{-t}},\quad s=\log (-t),
\]
and changes the sign of $V+z\cdot \nb_z V$ in \eqref{eq:Leray} to get the \emph{backward Leray equations}
\EQ{\label{eq:backLeray}
 \partial_s V  + \frac 12 V + \frac z2\cdot\nabla_z V -\Delta_z V+V\cdot\nabla_z V+\nabla_z \Pi  = 0,
 \quad \nabla_z\cdot V = 0.
}

The class of backward self-similar solutions, with $V=V(z)$ independent of $s$, was proposed by Leray \cite{leray} in the whole space as a possible ansatz for singularity. Such a possibility was excluded by \cite{NRS} when one assumes $V \in L^3(\R^3)$. A key ingredient in the proof of \cite{NRS} is that the quantity  (a modified ``total head pressure'')
\EQ{
\Lambda = \frac 12 |V|^2 + \Pi + \frac 12 y\cdot V
}
satisfies the 1-sided maximal principle. The result of \cite{NRS} was extended in \cite{Tsai-ARMA} to exclude the cases $V \in L^q(\R^3)$, $3< q \le \infty$. 

The result of \cite{NRS} also follows from the later paper \cite{ESS} by Escauriaza, Seregin, and Sverak, which says that any solution
$v \in L^\infty(-1,0; L^3(\R^3))$ is necessarily regular up to time $0$.
This result is extended in \cite{Seregin} to the half space.
See \cite{Barker-Seregin} and its references for other extensions. 
This result of \cite{ESS} can be also used to exclude RSS/DSS/RDSS solutions in the class $v \in L^\infty(-1,0; L^3(\R^3))$, but it does not imply the result of \cite{Tsai-ARMA}.
The following remains open. 

\begin{problem} Suppose $v(x,t)$ is a backward RSS/DSS/RDSS solution of \eqref{eq:NSE} in $\R^3$ with 
\EQ{
\label{S5:eq3}
|v(x,t)| \le \frac C{|x|+\sqrt{-t}} \quad \text{in }\R^3 \times (-1,0),
}
does $v$ remain bounded up to time $0$?
\end{problem}

The special case of RSS solutions of Perelman can be formulated as follows. (We could misunderstand since we did not communicate with Perelman directly.)

\begin{problem} Suppose $v(x,t)$ is a solution of \eqref{eq:NSE}  in $\R^3$ for $t<0$ with 
\EQ{
v(x,t) =  \frac 1{\sqrt{-t}} R(\al s) u(R(-\al s)\frac x{\sqrt t}),\quad s=\log (-t),
}
for some $\al \not =0$,
and
\EQ{\label{S5:eq2}
|u(y)| \le \frac C{|y|+1} \quad \text{in }\R^3.
}
Are $v$ and $u$ identically zero?
\end{problem}
One can study $u(y)$ directly, which satisfies \eqref{S5:eq2} and
\EQ{\label{eq:Perelman}
\al  J u -\al J y\cdot \nb u
  &+  \frac 12 u +  \frac y2\cdot\nabla u -\Delta_y u +u\cdot\nabla u+\nabla p  = 0,
 \\ & \nabla \cdot u = 0,
}
in $\R^3$.
Note that, for the above system with $\al \not =0$, there does not seem an analogue quantity of $\Lambda$ that satisfies the maximal principle.

Regarding the half space, such solutions in the class
$v \in L^\infty(-1,0; L^3(\R^3_+))$
do not exist by \cite{Seregin}, and nothing is known under the assumption \eqref{S5:eq3}. We formulate a problem.
\begin{problem} Suppose $v(x,t)$ is a solution of \eqref{eq:NSE} in $\R^3_+ \times (-\infty,0)$ with zero boundary condition and
\EQ{
v(x,t) =  \frac 1{\sqrt{-t}}  u(\frac x{\sqrt t}),
}
with $u(y)$ satisfying \eqref{S5:eq2}. Are $v$ and $u$ identically zero?
\end{problem}

\section*{Acknowledgments}
{Part of this work was done when Z.B. was a postdoctoral fellow at the University of British Columbia.}
The research of
both authors was partially supported by the NSERC grant 261356-13 (Canada). That of Z.B. was also partially supported by the
NSERC grant 251124-12.

Zachary Bradshaw, Department of Mathematics, University of Virginia, Charlottesville, VA, 22904,  USA;
e-mail: zb8br@virginia.edu

\medskip

Tai-Peng Tsai, Department of Mathematics, University of British
Columbia, Vancouver, BC V6T 1Z2, Canada;
e-mail: ttsai@math.ubc.ca

\end{document}